\documentclass[10pt]{amsart}
\usepackage{amsmath, amssymb}
\usepackage{mathrsfs}
\newcommand{\no}[1]{#1}
\renewcommand{\no}[1]{}
\no{\usepackage{times}\usepackage[subscriptcorrection, slantedGreek, nofontinfo]{mtpro}
\renewcommand{\Delta}{\upDelta}}
\usepackage{color}
\usepackage{enumerate,comment}

\date{\today}
\newtheorem{theorem}{Theorem}[section]
\newtheorem{proposition}{Proposition}[section]
\newtheorem{lemma}{Lemma}[section]

\newtheorem{corollary}{Corollary}[section]

\theoremstyle{remark}
\newtheorem{remark}{Remark}[section]

\newcommand{\bel}{\begin{equation} \label}
\newcommand{\ee}{\end{equation}}

\newcommand{\C}{{\mathbb C}}
\newcommand{\R}{{\mathbb R}}
\newcommand{\N}{{\mathbb N}}

\newcommand{\B}{{\mathcal B}}

\newcommand{\E}{{\mathcal E}}

\newcommand{\OO}{{\mathcal O}}

\newcommand{\U}{{\mathcal U}}

\newcommand{\Z}{{\mathbb Z}}

\def\beq{\begin{equation}}
\def\eeq{\end{equation}}
\newcommand{\bea}{\begin{eqnarray}}
\newcommand{\eea}{\end{eqnarray}}
\newcommand{\beas}{\begin{eqnarray*}}
\newcommand{\eeas}{\end{eqnarray*}}

\numberwithin{equation}{section}

\title[Stable determination of time-dependent and space periodic scalar potential]{Stable determination of time-dependent scalar potential from boundary measurements in a periodic quantum waveguide}

\author[Mourad Choulli]{Mourad Choulli}
\address{IECL, UMR CNRS 7502, Universit\'e de Lorraine, Ile du Saulcy, 57045 Metz cedex 1, France}
\email{mourad.choulli@univ-lorraine.fr}

\author[Yavar Kian]{Yavar Kian}
\address{CPT, UMR CNRS 7332, Aix Marseille Universit\'e, 13288 Marseille, France, and Universit\'e de Toulon, 83957 La Garde, France}
\email{yavar.kian@univ-amu.fr}

\author[Eric Soccorsi]{Eric Soccorsi}
\address{CPT, UMR CNRS 7332, Aix Marseille Universit\'e, 13288 Marseille, France, and Universit\'e de Toulon, 83957 La Garde, France}
\email{eric.soccorsi@univ-amu.fr}

\date{}

\begin{document}

\begin{abstract}
We prove logarithmic stability in the determination of the time-dependent scalar potential in a $1$-periodic quantum cylindrical waveguide, from the boundary measurements of the solution to the dynamic Schr\"odinger equation.


\end{abstract}

\maketitle

\tableofcontents

\section{Introduction}
\label{introduction}
\setcounter{equation}{0}

\subsection{Statement of the problem and existing papers}
In the present paper we consider an infinite waveguide $\Omega=\R \times \omega$, where 
$\omega$ is a bounded domain of $\mathbb{R}^2$ with $C^{2}$-boundary $\partial \omega$.
We assume without limiting the generality of the foregoing that $\omega$ contains the origin.
For shortness sake we write $x=(x_1,x')$ with $x'=(x_2,x_3)$ for every $x=(x_1,x_2,x_3) \in \Omega$. Given $T>0$, we consider the following initial-boundary value problem (abbreviated to IBVP in what follows)

\begin{equation}
\label{(1.1)}
\left\{
\begin{array}{ll}
(-i\partial _t -\Delta u +V(t,x))u=0\quad &\mbox{in}\; Q=(0,T)\times \Omega ,
\\
u(0,\cdot )=u_0 &\mbox{in}\; \Omega ,
\\
u=g &\mbox{on}\; \Sigma =(0,T)\times \partial \Omega ,
\end{array}
\right.
\end{equation}
where the electric potential $V$ is $1$-periodic with respect to the infinite variable $x_1$:
\begin{equation}\label{(1.2)}
V(\cdot  ,x_1+1,\cdot )=V(\cdot ,x_1,\cdot ),\;\; x_1\in\mathbb{R}.
\end{equation}
The main purpose of this paper is to prove stability in the recovery of $V$ from the ``boundary" operator
\bel{a1}
\Lambda_V:(u_0,g)\longrightarrow (\partial _\nu u_{| \Sigma}, u(T,\cdot)),
\ee
where the measure of $\partial_\nu u_{| \Sigma}$ (resp. $u(T,\cdot)$) is performed on $\Sigma$ (resp. $\Omega$). Here $\nu(x)$, $x \in \partial \Omega$, denotes the outward unit normal to $\Omega$ and $\partial_\nu u (t,x)= \nabla u (t,x) \cdot \nu(x)$.

There are only a few results available in the mathematical literature on the identification of time-dependent coefficients appearing in an initial boundary-value problem. G. Eskin proved in \cite{Esk7} that time analytic coefficients of hyperbolic equations are uniquely determined from the knwoledge of partial Neumann data. The case of a bounded cylindrical domain was addressed in \cite{GK} where the time-dependent coefficient of order zero in a parabolic equation is stably determined from a single Neumann boundary data. In \cite{Cho}[\S 3.6.3], using optics geometric solutions, M. Choulli proved logarithmic stability in the recovery of zero order time-dependent coefficients from partial boundary measurements for parabolic equations. In \cite{CK} Lipschitz stability was derived in the same problem for coefficients depending only on time from a single measurement of the solution. 

All the above mentioned results were obtained in bounded domains. Several authors considered the problem of recovering time independent
coefficients in an unbounded domain from boundary measurements. In most of the cases the unbounded domain under consideration is either an half space or an infinite slab. In \cite{Ra} Rakesh studied
the problem of recovering a scalar potential of the wave equation in an half
space from the Neumann-to-Dirichlet map. Then, by combining a unique continuation theorem for
the Cauchy problem associated to wave equation with constant speed with the
X-ray transform method developed by Hamaker, Smith, Solmon, Wagner in \cite{HSSW}, he
derived uniqueness for electric potentials which are constant outside some
{\it a priori} fixed compact set. This result has been extended by Nakamura to a more general class of coefficients in \cite{Na}. 
The inverse problem of identifying an embedded object in an infinite slab was examined in \cite{Ik} and \cite{SW}, 
and in \cite{LU} the compactly supported coefficients of a stationary
Schr\"odinger equation are identified from the knowledge of partial Dirichlet-to-Neumann (abbreviated to DN in the following) data. In an infinite cylindrical waveguide \cite{CS} and \cite{K} established
a stability estimate with respect to the DN map for an unknown coefficient 
in absence of any assumption on its behavior outside a compact set.
For inverse problems with time-independent coefficients in unbounded domains we refer to \cite{DKLS}.

In \cite{Esk8}, Eskin proved uniqueness modulo gauge invariance in the inverse problem of determining the time-dependent electric and magnetic potentials from the DN map for the Schr\"odinger equation in a simply-connected bounded or unbounded domain.
More specifically the inverse problem of determining periodic coefficients in the Helmholz equation was recently examined in \cite{Fl}.

\subsection{Main results and outline}
In order to express the main result of this article we first define the trace operator $\tau _0$ by
$$
\tau _0w = \left( w_{| \Sigma },w(0,\cdot )\right),\ w\in C_0^\infty ([0,T]\times \R,C^\infty(\overline{\omega})),
$$
and extend it to a bounded operator from $H^2(0,T;H^2(\Omega ))$ into $L^2((0,T)\times \mathbb{R},H^{3/2}(\partial \omega ))\times L^2(\Omega )$. 
The space $Y_0=H^2(0,T;H^2(\Omega ))/\mbox{Ker}(\tau _0)$, equipped with its natural quotient norm, is Hilbertian according to \cite{Sc}[\S XXIII.4.2, Theorem 2].
Moreover the mapping $\widetilde{\tau}_0 : Y_0 \ni \dot{W} \mapsto \tau _0 W \in \tau _0(H^2(0,T;H^2(\Omega )))$, where $W$ is arbitrary in $\dot{W}$, being bijective, it turns out that
$X_0=\tau _0(H^2(0,T;H^2(\Omega )))$ is an Hilbert space for the norm
$$
\|w\|_{X_0}=\|\widetilde{\tau}_0^{-1}w\|_{Y_0},\ w\in X_0.
$$
Furthermore we have
$$
\|w\|_{X_0}=\inf\{ \|W\|_{H^2(0,T;H^2(\Omega ))};\; W\in H^2(0,T;H^2(\Omega ))\; \mbox{such that}\; \tau_0W=w\}.
$$
As will be seen in the coming section, the linear operator $\Lambda _V$ defined by \eqref{a1} is actually bounded from $X_0$ 
to $X_1=L^2(\Sigma)\times L^2(\Omega )$.
Last, putting
\[
\Omega '=(0,1)\times \omega,\ Q'=(0,T)\times \Omega',\ \Sigma_*'=(0,T)\times (0,1)\times \partial \omega,
\]
we may now state the main result of this paper.

\begin{theorem}
\label{thm1} 
For $M>0$ fixed, let $V_1,\ V_2 \in W^{2,\infty}(0,T;W^{2,\infty}(\Omega))$ fulfill \eqref{(1.2)} together with the 
three following conditions:
\bel{(1.3)}
(V_2-V_1)(T,\cdot )=(V_2-V_1)(0,\cdot )=0\ \mbox{in}\  \Omega',
\ee
\bel{(1.4)}
V_2-V_1=0\ \mbox{in}\ \Sigma_*',
\ee
\bel{(1.5)}
\| V_j\|_{W^{2,\infty}(0,T;W^{2,\infty}(\Omega'))} \leq M,\ j=1,2.
\ee
Then there are two constants $C>0$ and $\gamma ^\ast>0$, depending only on $T$, $\omega$ and $M$, such that the estimate
$$
\| V_2-V_1 \|_{L^2(Q')}\leq C \left( \ln \left(\frac{1}{ \| \Lambda_{V_2}-\Lambda_{V_1} \|_{{\B(X_0,X_1)}}}\right)\right)^{-\frac{2}{5}},
$$
holds whenever $0 < \| \Lambda_{V_2}-\Lambda_{V_1} \|_{\B(X_0,X_1)} < \gamma ^\ast$.
\end{theorem}

Theorem \ref{thm1} follows from an auxiliary result we shall make precise below, which is related to the
following IBVP with quasi-periodic boundary conditions, 
 \begin{equation}
\label{(1.6)}
\left\{
\begin{array}{ll}
(-i\partial _t -\Delta  +V)v=0 &\mbox{in}\ Q',
\\
v(0,\cdot )=v_0 &\mbox{in}\ \Omega ',
\\
v=h &\mbox{on}\ \Sigma_*',
\\
v(\cdot ,1,\cdot )=e^{i\theta} v(\cdot ,0,\cdot) & \mathrm{on}\ (0,T) \times \omega,
\\
\partial_{x_1} v(\cdot ,1,\cdot )=e^{i\theta} \partial_{x_1} v(\cdot ,0,\cdot) & \mathrm{on}\ (0,T) \times \omega,
\end{array}
\right.
\end{equation}
where $\theta$ is arbitrarily fixed in $ [0,2 \pi)$. To this purpose we introduce the \footnote{The full definition of $H_{\sharp,\theta}^2(\Omega')$ can be found in section \ref{sec-FBG}.}{set}
$$H_{\sharp,\theta}^2(\Omega') =\{ u \in H^2(\Omega');\ u(1,\cdot)=e^{i \theta} u(0,\cdot)\ \mbox{and}\ \partial_{x_1} u(1,\cdot)=e^{i \theta} \partial_{x_1} u(0,\cdot)\ \mbox{on}\  \partial \omega \}, $$
and note $\tau_0'$ the linear bounded operator from $H^2(0,T;H^2(\Omega'))$ into $L^2((0,T)\times (0,1),H^{3/2}(\partial \omega ))\times L^2(\Omega')$, such that
$$
\tau'_0 w = \left(w_{| \Sigma_*'},w(0,\cdot)\right)\ \mathrm{for}\ w\in C_0^\infty ((0,T)\times (0,1),C^\infty(\overline{\omega})).
$$
Then the space $\mathscr{X}_{0,\theta}'=\tau_0'(H^2(0,T;H_{\sharp,\theta}^2(\Omega')))$, endowed with the norm
$$ \| w \|_{\mathscr{X}_{0,\theta}'} = \inf \{ \| W \|_{H^2(0,T;H^2(\Omega'))};\ W \in H^2(0,T;H_{\sharp,\theta}^2(\Omega'))\ \mathrm{satisfies}\ \tau_0' W = w \}, $$
is Hilbertian, and it is shown in section \ref{sec-rbo} that the operator
\[
\Lambda _{V,\theta }:(v_0,h)\in \mathscr{X}'_0 \mapsto (\partial _\nu v_\theta , v_\theta (T,\cdot )) \in \mathscr{X}_1'=L^2((0,T)\times (0,1) \times \partial \omega )\times L^2(\Omega'),
\]
is bounded. Here $v_\theta$ denotes the $L^2(0,T;H_{\sharp,\theta}^2(\Omega')) \cap H^1(0,T;L^2(\Omega'))$-solution of \eqref{(1.6)} associated to $(v_0,h)$. The following result essentially claims that Theorem \ref{thm1} remains valid upon substituting $\mathscr{X}'_{0,\theta}$ (resp.$\mathscr{X}'_1$) for $X_0$ (resp. $X_1$).

\begin{theorem}
\label{thm2} 
Let $M$ and  $V_j$, $j=1,2$, be the same as in Theorem \ref{thm1}. Then we may find two
constants $C>0$ and $\gamma ^\ast>0$, depending on $T$, $\omega$ and $M$, such that we have
\begin{equation}\label{(1.10)}
\| V_2-V_1 \|_{L^2(Q')}\leq C \left( \ln \left(\frac{1}{ \| \Lambda_{V_2,\theta}-\Lambda_{V_1,\theta} \|_{{\B(\mathscr{X}'_0,\mathscr{X}'_1)}}}\right)\right)^{-\frac{2}{5}},
\end{equation}
for any $\theta \in [0,2 \pi)$, provided
\begin{equation}\label{(1.11)}
0 < \| \Lambda_{V_2,\theta}-\Lambda_{V_1,\theta} \|_{\B(\mathscr{X}_{0,\theta}',\mathscr{X}_1'))}
< \gamma ^\ast .
\end{equation}
\end{theorem}

The text is organized as follows. In section \ref{sec-bo} we define the boundary operator $\Lambda_V$ and prove that it is bounded. In section \ref{sec-dec} we use the Floquet-Bloch-Gel'fand transform to decompose the IBVP \eqref{(1.1)} into a collection of problems \eqref{(1.6)} with quasi-periodic boundary conditions. section \ref{sec-ogs} is devoted to building optics geometric solutions to the above mentioned quasi-periodic boundary value problems. Finally the proof of Theorems \ref{thm1}-\ref{thm2}, which is by means of suitable optics geometric solutions defined in section \ref{sec-ogs}, is detailled in section \ref{sec-si}.

\begin{remark}
The method of the proofs of Theorems \ref{thm1} and \ref{thm2} given in the remaining part of this text can be easily adapted to the case of the inverse elliptic problem of recovering the (time-independent) periodic scalar potential $V$ in the stationnary Schr\"odinger equation 
$$ \left\{
\begin{array}{ll}
(-\Delta +V(x))u=0\quad &\mbox{in}\ \Omega,
\\
u=g &\mbox{on}\  \partial \Omega,
\end{array}
\right.
$$
from the knowledge of the DN map $g \mapsto \partial_\nu u$. Nevertheless, in order to prevent the inadequate expense of the size of this paper, we shall not extend the technique developped in the following sections to this peculiar framework.
\end{remark}


\section{Boundary operator}
\label{sec-bo}
In this section we prove that the boundary operator $\Lambda_V$ is bounded from $X_0$ into $X_1$. This preliminarily requires that the existence, uniqueness and smoothness properties of the solution to the IBVP \eqref{(1.2)} be appropriately established in Corollary \ref{cor-exiuni}.
To this end we start by proving the following technical lemma.
\begin{lemma}
\label{lemma2.1}
Let $X$ be a Banach space, $M_0$ be a m-dissipative operator in $X$ with dense domain $D(M_0)$
and $B \in C ([0,T] ,\B(D(M_0)))$. Then for all $v_0 \in D(M_0)$ and
$f\in C([0,T],X)\cap L^1(0,T;D(M_0))$ $($resp. $f\in W^{1,1}(0,T;X))$ there is a unique solution
$v \in Z_0=C([0,T], D(M_0))\cap C^1([0,T],X)$ to the following Cauchy problem
\begin{eqnarray}\label{(2.1)}
\left\{
\begin{array}{ll}
v'(t)=M_0 v(t)+ B(t)v(t)+f(t),
\\
v(0)=v_0,
\end{array}
\right.
\end{eqnarray}
such that
\begin{equation}\label{(2.2)}
\|v\|_{Z_0} =\|v\|_{C^0([0,T], D(M_0))}+\|v\|_{C^1([0,T],X)} \leq C(\|v_0\|_{D(M_0)}+\|f\|_\ast).
\end{equation}
Here $C$ is some positive constant depending only on $T$ and $\|B\|_{C ([0,T] ,\B(D(M_0)))}$, and 
$\|f\|_\ast$ stands for the norm $\| f \|_{C([0,T],X)\cap L^1(0,T;D(M_0))}$ $($resp. $\|f\|_{W^{1,1}(0,T;X)})$.
\end{lemma}
\begin{proof}
Put $Y=C([0,T], D(M_0))$ and define
$$
\begin{array}{lccc}
K:&  Y&\rightarrow  & Y
\\
& v&\mapsto &  \left[t \mapsto (Kv)(t)=\int_0^t S(t-s)B(s)v(s)ds\right],
\end{array}
$$
where $S(t)$ denotes the contraction semi-group generated by $M_0$.
The operator $K$ is well defined from \cite{CH}[Proposition 4.1.6] and we have
\begin{equation}\label{(2.3)}
\|Kv(t)\|_X\leq  tM \|v\|_Y,\ t\in [0,T],\ M=\|B\|_{C ([0,T] ,\B(D(M_0)))}.
\end{equation}
Therefore $K\in \B(Y)$ and we get
\begin{equation}\label{(2.4)}
\| K^n v(t)\|_X \leq \frac{t^nM^n}{n!} \|v\|_Y,\ t\in [0,T],
\end{equation}
by iterating \eqref{(2.3)}. Fix $F\in Y$ and put $\widetilde{K}v=Kv+F$ for all $v\in Y$. Thus, since
\[
\widetilde{K}^n v -\widetilde{K}^n w=K^n(v-w),\ v,w\in Y,\ n \in \mathbb{N},
\]
\eqref{(2.4)} entails that $\widetilde{K}^n$ is strictly contractive for some $n \in \mathbb{N}^*$. Hence $\widetilde{K}$ admits a unique fixed point in $Y$, which is the unique solution $v \in Y$ to the following Volterra integral equation 
\begin{equation}
\label{(2.5)}
v(t)=\int_0^tS(t-s)B(s)v(s)dt+F(t),\ t\in [0,T].
\end{equation}
As a consequence we have
\begin{equation}\label{(2.6)}
\|v\|_Y\leq e^{MT}\|F\|_Y,
\end{equation}
by Gronwall lemma.

The last step of the proof is to choose $F(t)=S(t)v_0+\int_0^tS(t-s)f(s)ds$ for $t\in [0,T]$ and to apply \cite{CH}[Proposition 4.1.6] twice, so we find out that $F\in Y$. Therefore the function $v$ given by \eqref{(2.5)} belongs to $C^1([0,T],X)$ and it is the unique solution to \eqref{(2.1)}. Finally we complete the proof by noticing that \eqref{(2.2)} follows readily from \eqref{(2.6)}.
\end{proof}

Prior to solving the IBVP \eqref{(1.1)} with the aid of Lemma \ref{lemma2.1}, we define the Dirichlet Laplacian $A_0=-\Delta^D$ in $L^2(\Omega)$ as the selfadjoint operator generated in $L^2(\Omega)$ by the closed quadratic form 
$$a_0(u)=\int_{\Omega} |\nabla u(x)|^2 dx,\ u \in D(a_0)=H_0^1(\Omega),$$
and establish the coming:

\begin{lemma}
\label{lemma2.2}
The domain of the operator $A_0$ is $H_0^1(\Omega) \cap H^2(\Omega)$ and the norm associated to $D(A_0)$ is equivalent to the usual one in $H^2(\Omega)$.
\end{lemma}

\begin{proof}
We have
\begin{equation}
\label{(2.7)}
\mathscr{F} A_0 \mathscr{F}^{-1} = \int_{\mathbb{R}}^{\oplus} \widehat{A}_{0,k} dk,
\end{equation}
where $\mathscr{F}$ denotes the partial Fourier with respect to $x_1$, i.e.
\[
(\mathscr{F} u)(k,x') = \widehat{u}(k,x')=\frac{1}{(2 \pi)^{1 \slash 2}} \int_{\mathbb{R}} e^{-i k x_1} u(x_1,x') dx_1,\ (k,x') \in \Omega, 
\]
and $\widehat{A}_{0,k}=-\Delta_{x'}+k^2$, $k \in \mathbb{R}$, is the selfadjoint operator in $L^2(\omega)$
generated by the closed quadratic form $\widehat{a}_{0,k}(v)=\int_{\omega} (|\nabla v(x')|^2 + k^2 |v(x')|^2 )dx'$, $v \in D(\widehat{a}_{0,k})=D(\widehat{a}_{0})=H_0^1(\omega)$. 
Since $\omega$ is a bounded domain with $C^2$-boundary, we have $D(\widehat{A}_{0,k})=H_0^1(\omega) \cap H^2(\omega)$ for each $k \in \mathbb{R}$, by \cite{Ag}.

Further, bearing in mind that \eqref{(2.7)} reads
\[
\left\{ \begin{array}{l}
D(A_0) = \{ u \in L^2(\Omega),\ \widehat{u}(k) \in D(\widehat{A}_{0,k})\ \textrm{a.e.}\ k \in \mathbb{R}\ \textrm{and}\ \int_{\mathbb{R}} \|\widehat{A}_{0,k} \widehat{u}(k) \|_{L^2(\omega)}^2 dk < \infty \},\\
(\mathscr{F} A_0 u)(k) = \widehat{A}_{0,k} \widehat{u}(k)\ \textrm{a.e.}\ k \in \mathbb{R}, \end{array} \right. 
\]
and noticing that
\[
 \|\widehat{A}_{0,k} v \|_{L^2(\omega)}^2= \sum_{j=0}^2 C_2^j k^{2j} \| \nabla^j v \|_{L^2(\omega)}^2,\ v \in H^1_0(\omega) \cap H^2(\omega),\ k \in \mathbb{R}, 
 \]
with $C_2^j=2! \slash (j!(2-j)!)$, $j=0,1,2$, we see that $D(A_0)$ is made of functions $u \in L^2(\Omega)$ satisfying simultaneously $\widehat{u}(k) \in H_0^1(\omega) \cap H^2(\omega)$ for a.e. $k \in \mathbb{R}$, and $k \mapsto (1+k^2)^{j \slash 2} \| \widehat{u}(k) \|_{H^{2-j}(\omega)} \in L^2(\mathbb{R})$ for $j=0,1,2$. Finally, $\| \widehat{u}(k) \|_{H^2(\omega)}$ being equivalent to $\| \Delta \widehat{u}(k) \|_{L^2(\omega)}$ by \cite{Ev}[\S 6.3, Theorem 4], we obtain the result.
\end{proof}

Let $B$ denote the multiplier by $V \in C([0,T],W^{2,\infty}(\Omega ))$. Due to Lemma \ref{lemma2.2} we have $B\in C([0,T],\B(D(A_0)))$ with $\|B\|_{C([0,T],\B(D(A_0)))}\leq \|V\|_{C([0,T],W^{2,\infty}(\Omega ))}$.
Therefore, applying Lemma \ref{lemma2.1} to $M_0=-iA_0$ we obtain the following existence and uniqueness result:

\begin{proposition}
\label{pr-exiuni}
Let $M>0$ and $V\in C([0,T],W^{2,\infty}(\Omega ))$ be such that $\|V\|_{C([0,T],W^{2,\infty}(\Omega ))} \leq M$. Then for all $v_0\in H_0^1(\Omega )\cap H^2(\Omega )$ and $f\in W^{1,1}(0,T;L^2(\Omega ))$ there is a unique solution $v \in Z_0=C([0,T],H_0^1(\Omega )\cap H^2(\Omega ))\cap C^1([0,T]; L^2(\Omega) )$ to 
\bel{b1}
\left\{
\begin{array}{ll}
-i \partial _t v -\Delta v +V v = f &\mbox{in}\ Q,
\\
v(0,\cdot )=v_0 &\mbox{in}\ \Omega ,
\\
v=0 &\mbox{on}\ \Sigma.
\end{array}
\right.
\ee
Moreover $v$ fulfills
$$
\|v\|_{Z_0}\leq C\left(\|v_0\|_{H^2(\Omega )}+\|f\|_{W^{1,1}(0,T; L^2(\Omega ))}\right),
$$
for some constant $C>0$ depending only on $\omega$, $T$ and $M$.
\end{proposition}
Finally, using a classical extension argument we now derive the coming useful consequence to Proposition \ref{pr-exiuni}.
\begin{corollary}
\label{cor-exiuni}
Let $M$ and $V$ be the same as in Proposition \ref{pr-exiuni}. Then for every $(g,u_0)\in X_0$, the IBVP \eqref{(1.1)} admits a unique \footnote{The coming proof actually establishes that this solution belongs to $C([0,T],H^2(\Omega ))\cap C^1([0,T],L^2(\Omega ))$.}{solution} 
$$
\mathfrak{s}(g,u_0)\in Z=L^2(0,T;H^2(\Omega ))\cap H^1(0,T;L^2(\Omega )).
$$
Moreover we have
\begin{equation}
\label{(2.8)}
\| \mathfrak{s}(g,u_0) \|_{Z}\leq C \| (g,u_0) \|_{X_0},
\end{equation}
for some constant $C>0$ depending only on $\omega$, $T$ and $M$.
\end{corollary}
\begin{proof}
Choose $G \in W^2(0,T;H^2(\Omega))$ obeying $\tau_0 G=(g,u_0)$ and $\| W \|_{W^2(0,T;H^2(\Omega))} \leq 2 \| (g,u_0) \|_{X_0}$. Then $u$ is solution to \eqref{(1.1)} if and only if $u-G$ is solution to \eqref{b1} with $f=i \partial_t G+\Delta G -V G$ and $v_0=u_0-G(0,.)$. Therefore the result follows from this and Proposition \ref{pr-exiuni}.
\end{proof}

Armed with Corollary \ref{cor-exiuni} we turn now to defining $\Lambda _V$. We preliminarily need to introduce the trace operator $\tau_1$, defined as the linear bounded operator from $L^2 ((0,T)\times \mathbb{R},H^2(\omega ))\cap H^1(0,T; L^2(\Omega ))$ into $L^2(\Sigma )\times L^2(\Omega )$, which coincides with the mapping
$$ w \mapsto \left( \partial _{\nu} w_{| \Sigma},w(T,\cdot)\right)\ \mathrm{for}\ w \in C_0^\infty ([0,T]\times \mathbb{R},C^\infty (\overline{\omega })). $$
Evidently, we have
$$
\|\tau _1 \mathfrak{s}(g,u_0) \|_{X_1}\leq C\| \mathfrak{s}(g,u_0) \|_Z \leq C\|(g,u_0)\|_{X_0},
$$
by \eqref{(2.8)}, hence the linear operator $\Lambda _V = \tau_1 \circ \mathfrak{s}$ is bounded from $X_0$ into $X_1$ with
$\|\Lambda _V\|=\|\Lambda _V\|_{\B(X_0,X_1)}\leq C$.
Here and the remaining part of this text $C$ denotes some suitable generic positive constant.

\begin{remark} 
In light of \cite{LM2}[Chap. 4,\S 2] and since $\overline{\Omega}$ is a smooth manifold with boundary $\partial \Omega$, we may as well define $\Lambda_V (g,u_0)$ in a similar way as before for all $u_0 \in H^2(\Omega)$ and all
$$g \in H^{3 \slash 2,3 \slash 2}(\Sigma)=L^2(0,T;H^{3 \slash 2}(\partial \Omega)) \cap H^{3 \slash 2}(0,T;L^2(\partial \Omega)), $$
fulfilling the compatibility conditions \cite{LM2}[Chap. 4, (2.47)-(2.48)].
 Nevertheless there is no need to impose these conditions in our approach since they are automatically verified by any $(g,u_0) \in X_0$.
\end{remark}


\section{Floquet-Bloch-Gel'fand analysis}
\label{sec-dec}
In this section we introduce the partial Floquet-Bloch-Gel'fand transform (abbreviated to FBG in the sequel) which is needed to decompose the Cauchy problem \eqref{(1.1)} into a collection of IBVP with quasi-periodic boundary conditions of the form \eqref{(1.6)}. 

\subsection{Partial FBG transform}
\label{sec-FBG}
The main tool for the analysis of periodic structures such as waveguides is the partial FBG transform defined for every 
$f \in C_0^\infty (Q)$ by
\bel{c1}
\check{f}_\theta(t,x)=(\mathcal{U} f)_\theta (t,x)=\sum_{k=-\infty}^{+\infty}e^{-ik\theta} f(t,x_1+k,x'),\ t \in \R,\ x=(x_1,x') \in \Omega,\ \theta \in [0,2 \pi).
\ee
We notice from \eqref{c1} that
\bel{c2}
\check{f}_\theta(t,x_1+1,x')=e^{i \theta} \check{f}_\theta(t,x_1,x'),\ t \in \R,\ x_1 \in \R,\ x' \in \omega,\ \theta \in [0,2 \pi),
\ee
and
\bel{c3}
\left( \U \frac{\partial^m f}{\partial z^m} \right)_\theta=\frac{\partial^m \check{f}_\theta}{\partial z^m},\ m \in \N^*,\ \theta \in [0,2 \pi),
\ee
whenever $z=t$ or $x_j$ , $j=1,2,3$. With reference to \cite{RS2}[\S XIII.16], $\mathcal{U}$ extends to a unitary operator, still denoted by $\mathcal{U}$, from $L^2(Q)$ onto the Hilbert space 
$\int_{(0,2 \pi)}^\oplus L^2(\Omega') d \theta \slash (2 \pi)=L^2( (0,2 \pi)d\theta \slash (2 \pi) , L^2((0,T) \times \Omega'))$.

Let $H_{\sharp,loc}^s(Q)$, $s=1,2$, denote the subspace of distributions $f$ in $Q$ such that $f_{| (0,T) \times I \times \omega} \in H^s((0,T) \times I \times \omega)$ for any bounded open subset $I \subset \R$. Then a function $f \in H^s_{\sharp,loc}(Q)$ is said to be $1$-periodic with respect to $x_1$ if $f(t,x_1+1,x')=f(t,x_1,x')$ for a.e. $(t,x_1,x') \in Q$. The subspace of functions of $H_{\sharp,loc}^s(Q)$ which are $1$-periodic with respect to $x_1$ is denoted by $H_{\sharp,per}^s(Q)$. Such a function is obviously determined by its values on $Q'$ so we set $H_{\sharp,per}^s(Q')=\{ u_{|Q'},\ u \in H_{\sharp,per}^s(Q) \}$. In light of \eqref{c2}-\eqref{c3}
we next introduce $H_{\sharp,\theta}^s(Q')= \{ e^{i \theta x_1} u; u \in H_{\sharp,per}^s(Q') \}$ for every $\theta \in [0,2 \pi)$. In view of \cite{Di}[Chap. II, \S 1, D\'efinition 1] we have
$$ \U H^s(Q)= \int_{(0,2 \pi)}^\oplus H_{\sharp,\theta}^s(Q') \frac{d \theta}{2 \pi},\ s=1,2. $$

More generally, for an arbirary open subset $Y$ of $\R^n$, $n \in \N^*$, we define the FBG transform with respect to $x_1$ of  $f \in C_0^{\infty}(\R \times Y)$ by
$$
\check{f}_{Y,\theta}(x_1,y)=(\U_Y f)_\theta (x_1,y)=\sum_{k=-\infty}^{+\infty}e^{-ik\theta} f(x_1+k,y),\ x_1 \in \R,\ y \in Y,\ \theta \in [0,2 \pi),
$$
and extend it to a unitary operator $\U_Y$ from $L^2(\R \times Y)$ onto $\int_{(0,2 \pi)}^\oplus L^2((0,1) \times Y) d\theta \slash (2 \pi)$. Similarly as before we say that a function $f \in H^s_{\sharp,loc}(\R \times Y)=H_{loc}^s(\R,L^2(Y)) \cap L_{loc}^2(\R,H^s(Y))$, $s>0$, is $1$-periodic with respect to $x_1$ if $f(x_1+1,y)=f(x_1,y)$ for a.e. $(x_1,y) \in \R \times Y$. Then we note $H_{\sharp,per}^s(\R \times Y)$ the subspace of $1$-periodic functions with respect to $x_1$ of $H_{\sharp,loc}^s(\R \times Y)$ and set $H_{\sharp,per}^s((0,1) \times Y)=\{ u_{|(0,1) \times Y},\ u \in H_{\sharp,per}^s(\R \times Y) \}$. 
Next we put $H_{\sharp,\theta}^s((0,1) \times Y)=\{ e^{i \theta x_1} u,\ u \in H_{\sharp,per}^s((0,1) \times Y) \}$ for all $\theta \in [0,2 \pi)$, so we have
$$ \U_Y H^s(\R \times Y)= \int_{(0,2 \pi)}^\oplus H_{\sharp,\theta}^s((0,1) \times Y) \frac{d \theta}{2 \pi},\ s >0. $$
For the sake of simplicity we will systematically omit the subscript $Y$ in $\U_Y$ in the remaining part of this text.

\subsection{FBG decomposition}
\label{sec-FBGde}
Bearing in mind that $\mathscr{X}_{0,\theta}'=\tau_0'(H^2(0,T;H_{\sharp,\theta}^2(\Omega')))$, where we recall that $\tau_0'$ is the linear bounded operator from $H^2(0,T;H^2(\Omega'))$ into $L^2((0,T)\times (0,1),H^{3/2}(\partial \omega ))\times L^2(\Omega')$ satisfying
$$
\tau'_0 w = \left(w_{| \Sigma_*'},w(0,\cdot)\right)\ \mathrm{for}\ w\in C_0^\infty ((0,T)\times (0,1),C^\infty(\overline{\omega})),
$$
it is apparent that
\bel{c7} 
\mathscr{X}_0 = \U X_0 = \int_{(0,2 \pi)}^{\oplus} \mathscr{X}_{0,\theta}' \frac{d \theta}{2 \pi}\ \mbox{and}\ \U \tau_0 \U^{-1} = \int_{(0,2 \pi)}^{\oplus} \tau_0' \frac{d \theta}{2 \pi}.
\ee
Here the notation $\tau_0'$ stands for the operator $\tau_0'$ restricted to $H^2(0,T;H_{\sharp,\theta}^2(\Omega'))$.
Further we put $\mathscr{Z}_\theta'= L^2(0,T;H_{\sharp,\theta}^2(\Omega')) \cap  H^1(0,T;L^2(\Omega'))$, $\theta \in [0,2 \pi)$,  so we have
\bel{c8}
\mathscr{Z} = \U Z = \int_{(0,2 \pi)}^\oplus \mathscr{Z}_\theta' \frac{d \theta}{2 \pi}.
\ee
Then, applying the transform $\U$ to both sides of each of the three lines in \eqref{(1.1)} we deduce from \eqref{c8} the following:
\begin{proposition}
\label{pr-equiv}
Let $V \in W^{2,\infty}(0,T;W^{2,\infty}(\Omega))$ fulfill \eqref{(1.2)} and let $(g,u_0) \in X_0$. Then $u$ is the solution $\mathfrak{s}(g,u_0) \in Z$ to \eqref{(1.1)} defined in Corollary \ref{cor-exiuni} if and only if $\U u \in \mathscr{Z}$ and
each $\check{u}_\theta=(\U u)_\theta \in \mathscr{Z}_\theta$, $\theta \in [0,2 \pi)$, is solution to the following IBVP
\bel{c10}
\left\{
\begin{array}{ll}
(-i\partial _t -\Delta  +V ) v =0\ &\mbox{in}\ Q'=(0,T)\times \Omega',
\\
v(0,\cdot)=\check{u}_{0,\theta} &\mbox{in}\ \Omega',
\\ 
v=\check{g}_\theta &\mbox{on}\ \Sigma_*',
\end{array}
\right.
\ee
where $\check{g}_\theta$  (resp. $\check{u}_{0,\theta}$) stands for $(\U g)_\theta$ (resp. $(\U u_0)_\theta$), ie $(\check{g}_\theta,\check{u}_{0,\theta}) =( \U (g,u_0))_\theta$.
\end{proposition}

\subsection{Reduced boundary operators}
\label{sec-rbo}
We first prove the following existence and uniqueness result for \eqref{c10} by arguing in the same way as in the derivation of Corollary \ref{cor-exiuni}.

\begin{lemma}
\label{lm-c}
Assume that $V$ obeys the conditions of Proposition \ref{pr-equiv} and satisfies $\| V \|_{W^2(0,T;W^{2,\infty}(\Omega'))} \leq M$ for some $M>0$. Then for every $(\check{g}_\theta,\check{u}_{0,\theta}) \in \mathscr{X}_{0,\theta}'$, $\theta \in [0,2 \pi)$, there exists a unique solution $\mathfrak{s}_\theta(\check{g}_\theta,\check{u}_{0,\theta}) \in \mathscr{Z}_\theta'$ to \eqref{c10}. Moreover we may find a constant $C=C(T,\omega,M)>0$ such that the estimate
\bel{c11} 
\| \mathfrak{s}_\theta(\check{g}_\theta,\check{u}_{0,\theta}) \|_{\mathscr{Z}_\theta'} \leq C \| (\check{g}_\theta,\check{u}_{0,\theta}) \|_{\mathscr{X}_{0,\theta}'},
\ee
holds for every $\theta \in [0,2 \pi)$.
\end{lemma}
\begin{proof}
Let $A_{0,\theta}$ be the selfadjoint operator in $L^2(\Omega')$ generated by the closed quadratic form 
$$ a_{0,\theta}(u)=\int_{\Omega'} | \nabla u(x) |^2 dx,\ u \in D(a_{0,\theta}) = L^2(0,1;H_0^1(\omega)) \cap H_{\sharp,\theta}^1(0,1;L^2(\omega)), $$
in such a way that $A_{0,\theta}$ acts as $(-\Delta)$ on its \footnote{This can be easily seen by arguing in the exact same way as in the derivation of Lemma \ref{lemma2.2}.}{domain} $D(A_{0,\theta})=H_{\sharp,\theta}^2(\Omega') \cap L^2(0,1;H_0^1(\omega))$. Let $B$ denote the multiplier by $V \in C([0,T],W^{2,\infty}(\Omega'))$. Due to \eqref{(1.2)} 
we have $B \in C([0,T],\B(D(A_{0,\theta})))$ and $\| B \|_{C([0,T],\B(D(A_{0,\theta})))} \leq \|V \|_{C([0,T],W^{2,\infty}(\Omega'))}$. Therefore, for every $f \in W^{1,1}(0,T;L^2(\Omega'))$ and $v_{0,\theta} \in H_{\sharp,\theta}^2(\Omega') \cap L^2(0,1;H_0^1(\omega))$ there is unique solution
$v_\theta \in L^2(0,T;L^2(0,1;H_0^1(\omega)) \cap H_{\sharp,\theta}^2(\Omega')) \cap H^1(0,T;L^2(\Omega'))$ to the IBVP
\bel{c12}
\left\{
\begin{array}{ll}
-i \partial _t v -\Delta v+V v = f &\mbox{in}\ Q',
\\
v(0,\cdot)=v_{0,\theta} &\mbox{in}\ \Omega' ,
\\
v=0 &\mbox{on}\ \Sigma_*',
\end{array}
\right.
\ee
by Lemma \ref{lemma2.1}, satisfying
$$
\| v_\theta \|_{L^2(0,T;H_{\sharp,\theta}^2(\Omega'))} + \| v_\theta \|_{H^1(0,T;L^2(\Omega'))} \leq C \left(\|v_{0,\theta}\|_{H^2(\Omega' )}+\|f\|_{W^{1,1}(0,T; L^2(\Omega' ))}\right).
$$
Further, from the very definition of  $\mathscr{X}_{0,\theta}'$ we may find $W_\theta \in H^2(0,T;H_{\sharp,\theta}^2(\Omega'))$ such that $\tau_{0,\theta}' W_\theta = (\check{g}_\theta,\check{u}_{0,\theta})$ and $\| W_\theta \|_{H^2(0,T;H^2(\Omega'))} \leq 2 \| (\check{g}_\theta, \check{u}_{0,\theta}) \|_{\mathscr{X}_{0,\theta}}$. Thus, taking $f=i \partial_t W_\theta + \Delta W_\theta - V W_\theta$ and $v_{0,\theta}=u_{0,\theta}-W_\theta(0,\cdot)$ in \eqref{c12}, it is clear that $v_\theta - W_\theta$ is solution to \eqref{c10} if and only if $v_\theta$ is solution to \eqref{c12}. This yields the desired result.
\end{proof}
In virtue of Lemma \ref{lm-c} the linear operator $\mathfrak{s}_\theta$ is thus bounded from $\mathscr{X}_{0,\theta}'$ into $\mathscr{Z}_\theta'$, with 
\bel{c13}
\| \mathfrak{s}_\theta \|=\| \mathfrak{s}_\theta \|_{\B(\mathscr{X}_{0,\theta}',\mathscr{Z}_\theta')} \leq C,\ \theta \in [0,2 \pi).
\ee
Let $\tau_1'$ be the linear bounded operator from $L^2((0,T) \times (0,1),H^2(\Omega')) \cap H^1(0,T;L^2(\Omega'))$ into $\mathscr{X}_1'=L^2((0,T)\times (0,1) \times \partial \omega )\times L^2(\Omega')$, obeying
$$
\tau_1' w = \left(\partial_\nu w_{|\Sigma_*'},w(T,\cdot)\right)\ \mathrm{for}\ w\in C_0^\infty ((0,T)\times (0,1),C^\infty(\overline{\omega})),
$$
in such a way that
\bel{c15}
 \mathscr{X}_1 = \U X_1 = \int_{(0,2 \pi)}^{\oplus} \mathscr{X}_{1}' \frac{d \theta}{2 \pi}\ \mathrm{and}\ \U \tau_1 \U^{-1} = \int_{(0,2 \pi)}^{\oplus} \tau_1' \frac{d \theta}{2 \pi},
\ee
In light of \eqref{c11} we then have
$$ \| \tau_1' \mathfrak{s}_\theta (\check{g}_\theta,\check{u}_{0,\theta}) \|_{\mathscr{X}_1'} \leq  C \| \mathfrak{s}_\theta (\check{g}_\theta,\check{u}_{0,\theta}) \|_{\mathscr{Z}_\theta'} \leq C \| (\check{g}_\theta,\check{u}_{0,\theta}) \|_{\mathscr{X}_{0,\theta}'},\ \theta \in [0,2 \pi), $$
so the reduced boundary operator $\Lambda_{V,\theta}=\tau_1' \circ \mathfrak{s}_\theta \in \B(\mathscr{X}_{0,\theta}',\mathscr{X}_1')$.

Further, bearing in mind \eqref{c7}-\eqref{c8} and \eqref{c15} we deduce from Proposition \ref{pr-equiv} and Lemma \ref{lm-c} that
$$ \U \Lambda_V \U^{-1} = \int_{(0,2 \pi)}^\oplus \Lambda_{V,\theta} \frac{d \theta}{2 \pi}. $$
In light of \cite{Di}[Chap. II, \S 2, Proposition 2] this finally entails that
\bel{c20}
\| \Lambda_V \|_{\B(X_0,X_1)} = \sup_{\theta \in (0,2 \pi)} \| \Lambda_{V,\theta} \|_{\B(\mathscr{X}_{0,\theta}',\mathscr{X}_1')}.
\ee

\section{Optics geometric solutions}
\label{sec-ogs}

Let $r>0$ and $\theta \in [0,2 \pi)$ be fixed. This section is devoted to building optics geometric solutions to the system
\begin{equation}
\label{(4.1)}
\left\{
\begin{array}{ll}
(-i\partial _t -\Delta  +V)v=0 &  \mbox{in}\; Q',
\\
u(\cdot ,1,\cdot )=e^{i\theta} u(\cdot ,0,\cdot) &  \mathrm{on}\ (0,T) \times \omega,
\\
\partial_{x_1}u(\cdot ,1,\cdot )=e^{i\theta}\partial_{x_1}u(\cdot ,0,\cdot) & \mathrm{on}\ (0,T) \times \omega.
\end{array}
\right.
\end{equation}
Specifically, we seek solutions $u_{k,\theta}$, $k \in \Z$, to \eqref{(4.1)} of the form
\[
u_{k,\theta}(t,x)=\left( e^{i \theta x_1}+w_{k,\theta}(t,x) \right)e^{-i\left( (\tau+ 4 \pi^2 k^2)t+x'\cdot\xi_j+2 \pi k x_1 \right)},\ x=(x_1,x') \in \Omega ',
\]
where $w_{k,\theta} \in H^2(0,T;H_{\sharp,\theta}^{2}(\Omega '))$ fulfills
$$
\| w_{k,\theta} \|_{H^2(0,T;H^{2}(\Omega '))}\leq \frac{c}{r}(1+|k|), 
$$
for some constant $c>0$ which is independent of $r$, $k$ and $\theta$.
The main issue here is the quasi-periodic condition imposed on $w_{k,\theta}$. To overcome this problem we shall adapt  the framework introduced in \cite{Ha} for defining optics geometric solutions in periodic media.

\subsection{Optics geometric solutions in periodic media}
Fix $R>0$ and put $\OO=(-R,R) \times(0,1) \times (-R,R)^2$. We recall that $u \in H^1_{loc}(\R^4)$ is $\OO$-periodic if it satisfies
$$
u(y+2R \E_j)=u(y),\ j=0,2,3,\ \textrm{and}\ u(y+ \E_1)=u(y),\ \textrm{a.e.}\ y=t \E_0 + \sum_{j=1}^3 x_j \E_j \in \OO, 
$$
where $\{ \E_j \}_{j=0}^3$ denotes the canonical basis of $\R^4$. We note $H^1_{per}(\OO)$ the subset of $\OO$-periodic functions in $H^1_{loc}(\R^4)$, endowed with the scalar product of $H^1(\OO)$. 
Similarly we define $H_{per}^2(\OO)=\{ u \in H_{per}^1(\OO),\ \partial_k u \in H_{per}^1(\OO),\ k=0,1,2,3 \}$.

Further we introduce the space
$$
\mathscr{H}_{\theta}= \{ e^{i \theta x_1} e^{i\frac{\pi x_2}{2R}} u;\; u \in H_{loc}^2(\R_t;H_{loc}^2(\R^3)) \cap H_{per}^2(\OO) \},\ \theta \in [0,2 \pi),
$$
which is Hilbertian for the natural norm of $\mathscr{H}^2=H^2(-R,R;H^{2}( (0,1) \times (-R,R)^2 ))$, and mimmick the proof of \cite{Ha}[Theorem 1] or \cite{Cho}[Proposition 2.19] to claim the coming technical result.

\begin{lemma}
\label{l2}
Let $s>0$, let $\kappa \in \R^4$ be such that $\kappa \cdot \E_2=0$, and set $\vartheta=s \E_2+ i \kappa$. Then for every $h \in \mathscr{H}^2$ the equation
\bel{l2a}
-i\partial_t \psi-\Delta \psi+2 \vartheta \cdot \nabla\psi=h\ \textrm{in}\ \OO,
\ee
admits a unique solution $\psi \in \mathscr{H}_{\theta}$. Moreover, it holds true that
\bel{l2b}
\| \psi \|_{\mathscr{H}^2} \leq  \frac{R}{s\pi} \| h \|_{\mathscr{H}^2}.
\ee
\end{lemma}
\begin{proof} 
For all $\alpha \in \Z_\theta=\theta \E_1+\frac{\pi}{2R} \E_2+\left(\frac{\pi}{R} \Z\right) \times \Z \times \left(\frac{\pi}{R} \Z\right)^2$, put
$$
 \phi_\alpha(y)=\frac{1}{(2R)^{\frac{3}{2}}}e^{i\alpha\cdot y},\ y=(t,x) \in \OO, 
$$
in such a way that $\{ \phi_{\alpha} \}_{\alpha \in \Z_{\theta}}$ is a Hilbert basis of $L^2(\OO)$. Assume that $\psi\in \mathscr{H}_{\theta}$ is solution to \eqref{l2a}.
Then for each $\alpha \in \Z_\theta$ it holds true that $\langle h , \phi_\alpha \rangle_{L^2(\OO)} = \langle -i\partial_t \psi-\Delta \psi+2 \vartheta \cdot \nabla\psi,\phi_\alpha \rangle_{L^2(\OO)}$ whence 
\beas
\langle h , \phi_\alpha \rangle_{L^2(\OO)} & =  & \langle \psi,-i\partial_t \phi_\alpha-\Delta \phi_\alpha-2\overline{\vartheta} \cdot \nabla \phi_\alpha \rangle_{L^2(\OO)} 
\\
& = &
\left(\alpha_0+\sum_{j=1}^3 \alpha_j^2-2\kappa \cdot\alpha +2is\alpha_2 \right) \langle \psi,\phi_\alpha \rangle_{L^2(\OO)},
\eeas
by integrating by parts, with
\bel{eq2b}
\left| \Im\left({\alpha_0+\sum_{j=1}^3 \alpha_j^2-2 \kappa \cdot\alpha+2is\alpha_2}\right) \right| = 2 s | \alpha_2 | \geq\frac{s\pi}{R}.
\ee
Therefore we necessarily have
\begin{equation}
\label{e1}
\psi=\sum_{\alpha\in {Z}_\theta}\frac{\langle h, \phi_\alpha \rangle_{L^2(\OO)}}{\alpha_0+\sum_{j=1}^3\alpha_j^2-2 \kappa \cdot\alpha+2is\alpha_2}\phi_\alpha.
\end{equation}
On the other hand the function $\psi$ defined by the right hand side of \eqref{eq2b} is in $\mathscr{H}^2$ since 
\beas
\| \psi \|_{\mathscr{H}^2}^2 & =& \sum_{\alpha \in \Z_{\theta}} \frac{(1+\alpha_0^2+\alpha_0^4) \left( \sum_{1 \leq j,l \leq 3} \alpha_j^2 \alpha_l^2 \right) | \langle h,\phi_\alpha \rangle_{L^2(\OO)} |^2}{\left| \alpha_0+\sum_{j=1}^3 \alpha_j^2-2 \kappa \cdot\alpha +2is\alpha_2\right|^2} \\
& \leq & \frac{R^2}{s^2 \pi^2} \sum_{\alpha \in \Z_{\theta}} (1+\alpha_0^2+\alpha_0^4) 
\left( \sum_{1 \leq j,l \leq 3} \alpha_j^2 \alpha_l^2 \right) | \langle h,\phi_{\alpha} \rangle_{L^2(\OO)} |^2 < +\infty.
\eeas
Here we used the fact that the last sum over $\alpha \in \Z_{\theta}$ is equal to $\| h \|_{\mathscr{H}^2}^2$, which incidentally entails \eqref{l2b}. Finally the trace operators $w \mapsto \partial_{x_1}^m w_{|[-R,R]\times \{ 0,1 \} \times [-R,R]^2}$ being continuous on $\mathscr{H}^2$ for $m=0,1$, we end up getting that $\psi \in \mathscr{H}_\theta$.
\end{proof}

\begin{remark}
It should be noticed that in contrast to \cite{Ha}[Theorem 1] where the fundamental $H^2$-solutions $\psi$ to the Faddeev-type equation are obtained from any $L^2$-right hand side $h$, it is actually required in Lemma \ref{l2} that $h$ be taken in $\mathscr{H}^2$. This boils down to the fact that the elliptic regularity of the Faddeev equation does not hold for the Schr\"odinger equation \eqref{l2a}.
\end{remark}

\subsection{Building optics geometric solutions}
We first deduce from Lemma \ref{l2} the:

\begin{lemma}
\label{l3}  
Let $\xi \in \C^2 \setminus \R^2$ verify
\bel{llc}
\Im{\xi} \cdot \Re{\xi}=0.
\ee 
Then, for all $\theta \in [0,2 \pi)$ and $k \in \Z$, there exists $E_{k,\theta} \in \B(H^2(0,T; H^2(\Omega')), H^2(0,T; H_{\sharp,\theta}^2(\Omega')))$
such that $\varphi=E_{k,\theta} f$, where $f \in  H^2(0,T; H^2(\Omega'))$, is solution to the equation
\bel{l3a}
(-i\partial_t - \Delta +4i \pi k \partial_{x_1}+2 i \xi \cdot \nabla_{x'}) \varphi =f\; \mbox{in}\; Q'.
\ee
Moreover we have
\bel{l3b}
\| E_{k,\theta} \|_{\B(H^2(0,T;H^{2}(\Omega ')))}\leq \frac{c_0}{| \Im{\xi} |},
\ee
for some constant $c_0>0$ independent of $\xi$, $k$ and $\theta$.
\end{lemma}

\begin{proof}
Pick $R>0$ so large that any planar rotation around the origin of $\R^2$ maps $\omega$ into $(-R,R)^2$. Next, bearing in mind that $r=| \Im \xi| >0$, we call $S$ the unique planar rotation around $0_{\R^2} \in \omega$, mapping
the second vector $\mathfrak{e}_2$ in the canonical basis of $\R^2$ onto $-\Im \xi \slash r$:
\bel{eq3} 
S \mathfrak{e}_2 = -\frac{\Im \xi}{r}.
\ee
Further, pick $f\in H^2(0,T; H^2(\Omega '))$, and put
\bel{eq3b}
 \widetilde{f}(t,x_1,x') = f(t,x_1,S^* x'),\ (t,x_1,x') \in (0,T) \times (0,1)\times S \omega  , 
 \ee
where $S^*$ denotes the inverse transformation to $S$. Evidently, $\widetilde{f} \in H^2(0,T;H^{2}((0,1)\times S \omega ))$. Moreover, as $\partial \omega$ is $C^2$, there exists 
\[
 P \in \B(H^2(0,T;H^{2}((0,1)\times S \omega )), H^2(\R,H^{2}((0,1)\times \R^2 ))),
\]
such that $(P \tilde{f})_{\vert (0,T) \times (0,1)\times S \omega }=\widetilde{f}$, by \cite{LM1}[Chap. 1, Theorems 2.2 \& 8.1]. Let $\chi =\chi (t,x')\in C^\infty_0((-R,R)^3)$ fulfill $\chi=1$ in a neighborhood of $[0,T] \times \overline{S \omega}$. Then the function
\bel{eq3c}
h(t,x_1,x')=\chi(t,x')(P\tilde{f})(t,x_1,x'),\; (t,x_1,x')\in \OO,
\ee
belongs to $\mathscr{H}^2$. Moreover it holds true that 
\[
h_{| (0,T)\times (0,1)\times S \omega}=\widetilde{f}.
\]

The next step of the proof is to choose $\kappa=(0,2 \pi k,S^*\Re \xi ) \in \R^4$ so we get
\[
\kappa \cdot \E_2 = S^* \Re \xi \cdot \mathfrak{e}_2 = \Re \xi \cdot S \mathfrak{e}_2 = -\frac{\Re \xi \cdot \Im \xi}{r}=0, 
\]
by combining \eqref{llc} with \eqref{eq3}. We call $\psi$ the $\mathscr{H}_\theta$-solution to
\eqref{l2a} obtained by applying Lemma \ref{l2} with $\vartheta=r \E_2 + i \kappa$ and $h$ given by \eqref{eq3}--\eqref{eq3c}, and put
\bel{eq3d}
(E_{k,\theta} f)(t,x_1,x')=\psi(t,x_1,Sx'),\ (t,x_1,x')\in Q'.
\ee
Obviously, $E_{k,\theta} f \in H^2(0,T;H_{\sharp,\theta}^{2}(\Omega'))$ and \eqref{l2b} yields 
\bel{eq4}
\| E_{k,\theta} f \|_{H^2(0,T;H^{2}(\Omega '))} \leq C\| \psi \|_{\mathscr{H}^2} \leq \frac{CR}{r\pi} \| h \|_{\mathscr{H}^2},
\ee
Furthermore, in light of \eqref{eq3b}-\eqref{eq3c}
we have
\[
 \| h \|_{\mathscr{H}^2} \leq \| P \tilde{f} \|_{H^2(\R,H^{2}(\R^2\times(0,1)))}
\leq \| P \| \| \tilde{f} \|_{H^2(0,T;H^{2}((0,1)\times S \omega ))}
\leq C\| f \|_{H^2(0,T;H^{2}(\Omega '))}, 
\]
where $\| P \|$ stands for the norm of $P$ in the space of linear bounded operators acting from $H^2(0,T;H^{2}((0,1)\times S \omega ))$ into $H^2(\R,H^{2}((0,1)\times \R^2 ))$.
Putting this together with \eqref{eq4}, we end up getting \eqref{l3b}. 

This being said, it remains to show that $\varphi=E_{k,\theta}f$ is solution to \eqref{l3a}.
To see this we notice from \eqref{eq3d} that
$\varphi = \psi \circ F$, where $F$ is the unitary transform $(t,x_1,x') \mapsto (t,x_1,S x')$ in $\R^4$. As a consequence we have $\nabla \varphi= F \nabla \psi \circ F$, whence 
\bel{eq5}
\vartheta \cdot \nabla \psi \circ F = F \vartheta \cdot \nabla \varphi = i 2 \pi k \partial_{x_1} \varphi +i \xi \cdot \nabla_{x'} \varphi ,
\ee
and 
\bel{eq5b}
 -\Delta \varphi = -\nabla \cdot \nabla \varphi = -F \nabla \cdot F \nabla \psi \circ F =
-\nabla \cdot \nabla \psi \circ F = -\Delta \psi \circ F.
\ee
Moreover we have $h \circ F=\tilde{f} \circ F =f$ in $Q'$, directly from \eqref{eq3b}-\eqref{eq3c}, and 
$\partial_t \varphi = \partial_t \psi \circ F$, so \eqref{l3a} follows readily from this, \eqref{l2a} and \eqref{eq5}-\eqref{eq5b}.
\end{proof}

Armed with Lemma \ref{l3} we are now in position to establish the main result of this section.

\begin{proposition}
\label{pr1}
Assume that $V \in W^{2,\infty}(0,T;W^{2,\infty}(\Omega))$ satisfies \eqref{(1.2)} and $\| V \|_{W^{2,\infty}(0,T;W^{2,\infty}(\Omega))} \leq M$ for some $M \geq 0$.
Pick $r \geq r_0=c_0 (1+ M)$, where $c_0$ is the same as in \eqref{l3b}, and let $\xi \in \C^2 \setminus \R^2$ fulfill \eqref{llc} and
$| \Im \xi |=r$.
Then for all $\theta \in [0,2 \pi)$ and $k \in \Z$, there exists $w_{k,\theta}\in H^2(0,T;H_{\sharp,\theta}^{2}(\Omega '))$ obeying
\bel{t2a}
\| w_{k,\theta} \|_{H^2(0,T;H^{2}(\Omega '))} \leq \frac{c}{r} \left( 1+|k| \right),
\ee
for some constant $c>0$\footnote{Actually $c$ depends only on $T$, $|\omega|$ and $M$.} {independent} of $r$, $k$ and $\theta$, such that the function
\bel{t2r}
u_{k,\theta}(t,x)=\left(e^{i\theta x_1}+w_{k,\theta}(t,x) \right) e^{-i \left((\xi \cdot \xi+ 4 \pi^2 k^2)t+2 \pi k x_1+x'\cdot\xi \right)},\ (t,x)=(t,x_1,x') \in Q',
\ee
is a $H^2(0,T;H_{\sharp,\theta}^2(\Omega'))$-solution to the equation \eqref{(4.1)}.
\end{proposition}

\begin{proof} 
A direct calculation shows that $u_{k,\theta}$ fulfills \eqref{(4.1)} if
and only if $w_{k,\theta}$ is solution to 
\bel{t2b}
\left\{
\begin{array}{ll}
(-i\partial_t  - \Delta  + 4 i\pi k \partial_{x_1}+ 2i  \xi \cdot \nabla_{x'} + V)w + e^{i\theta x_1} W_{k,\theta} = 0 & \mathrm{in}\ Q', \\
w (\cdot ,1,\cdot )=e^{i\theta}w (\cdot ,0,\cdot) & \mathrm{on}\ (0,T) \times \omega,  \\
\partial_{x_1}w (\cdot ,1,\cdot )=e^{i\theta}\partial_{x_1}w (\cdot ,0,\cdot) & \mathrm{on}\ (0,T) \times \omega,
\end{array}
\right.
\ee
with
\bel{t2bb}
W_{k,\theta}=V + \theta^2 - 4 \pi k \theta.
\ee
In light of \eqref{t2b}-\eqref{t2bb} we introduce the map 
$$
\begin{array}{lccc}
G_{k,\theta}  :  & H^2(0,T;H_{\sharp,\theta}^{2}(\Omega')) & \longrightarrow & H^2(0,T;H_{\sharp,\theta}^{2}(\Omega'))\\
&  q & \longmapsto & -E_{k,\theta} \left( V q + e^{i \theta x_1} W_{k,\theta} \right),
\end{array}
$$
set
\bel{eq6a}
M = 12 \pi^2(3 T | \omega |)^{1 \slash 2}(4 \pi^2 + \| V \| + 8 \pi^2 |k|) ,
\ee
where $\| V \|$ is a shorthand for $\| V \|_{W^{2,\infty}(0,T;W^{2,\infty}(\Omega))}$, and notice that
\bel{eq6b}
\left\| W_{k,\theta} e^{i \theta x_1} \right\|_{H^2(0,T;H^{2}(\Omega '))} \leq M.
\ee
Then we have
$$
 \| G_{k,\theta} q \|_{H^2(0,T;H^{2}(\Omega'))} \leq \frac{c_0}{r} \left( \| V \| \| q \|_{H^2(0,T;H^{2}(\Omega'))} + M \right),\ q \in H^2(0,T;H_{\sharp,\theta}^{2}(\Omega ')),
$$
in virtue of \eqref{l3b} and \eqref{eq6b} .
From this and the condition $r \geq r_0$, involving,
\bel{eq6c}
r=|\Im \xi| \geq c_0(1+\| V \|),
\ee
then follows that $\| G_{k,\theta} q \|_{H^2(0,T;H^{2}(\Omega_*))} \leq M$
for all $q$ taken in the ball $B_{M}$ centered at the origin with radius $M$ in $H^2(0,T;H_{\sharp,\theta}^{2}(\Omega'))$. Moreover, it holds true that
$$
\| G_{k,\theta} q-G_{k,\theta} \tilde{q} \|_{H^2(0,T;H^{2}(\Omega'))} \leq \frac{\| q-\tilde{q} \|_{H^2(0,T;H^{2}(\Omega '))}}{2},\ q, \tilde{q}\in B_{M}, 
$$
hence $G_{k,\theta}$ has a unique fixed point $w_{k,\theta} \in H^2(0,T;H_{\sharp,\theta}^{2}(\Omega'))$. Further, by applying Lemma \ref{l3} with
\[ 
f= -\left( V w_{k,\theta} +  e^{i \theta x_1} W_{k,\theta} \right)\in H^2(0,T;H^{2}(\Omega')), 
\]
we deduce from \eqref{l3a} that $w_{k,\theta}=E_{k,\theta} f$ is a solution to 
\eqref{t2b}. Last, taking into account the identity $\| w_{k,\theta} \|_{H^2(0,T;H^{2}(\Omega'))} = \| G_{k,\theta} w_{k,\theta} \|_{H^2(0,T;H^{2}(\Omega '))}$, we get that
\beas
\| w_{k,\theta} \|_{H^2(0,T;H_\theta^2(\Omega'))} & \leq & \| G_{k,\theta} w_{k,\theta} -G_{k,\theta} \underline{0} \|_{H^2(0,T;H^{2}(\Omega '))}+ \| G_{k,\theta} \underline{0} \|_{H^2(0,T;H^{2}(\Omega '))}
\\
& \leq & \| E_{k,\theta}(V w_{k,\theta}) \|_{H^2(0,T;H^{2}(\Omega '))} +
\left\| E_{k,\theta} \left( e^{i \theta x_1} W_{k,\theta} \right) \right\|_{H^2(0,T;H^{2}(\Omega '))} 
\\
& \leq & \frac{c_0}{r} \left( \| V \| \| w_{k,\theta} \|_{H^2(0,T;H^{2}(\Omega '))} + M \right), 
\eeas
directly from \eqref{l3b} and \eqref{eq6b}. Here $\underline{0}$ denotes the function which is identically zero in $\Omega'$.
From this and \eqref{eq6c} then follows that
$\| w_{k,\theta} \|_{H^2(0,T;H^{2}(\Omega '))} \leq (2c_0 \slash r) M$, which, combined with \eqref{eq6a}, entails \eqref{t2a}.
\end{proof}

\section{Stability estimate}
\label{sec-si}
This section contains the proof of Theorem \ref{thm2}, which, along with \eqref{c20}, yields Theorem \ref{thm1}. We start by estabilishing two auxiliary results.

\subsection{Auxiliary results}
In view of deriving Lemma \ref{l4} from Proposition \ref{pr1}, we first prove the following technical result.
\begin{lemma}
\label{l1}
For all $r>0$ and $\zeta=(\eta,\ell) \in \R^2 \times\R$ with $\eta \neq 0_{\R^2}$, there exists $\zeta_j=\zeta_j(r,\eta,\ell)=(\xi_j,\tau_j) \in\C^2\times\R$, $j=1,2$, such that we have
\bel{lla}
| \Im {\xi_j} |=r,\ \ \tau_j=\xi_j\cdot\xi_j,\ \ \zeta_1-\overline{\zeta_2}=\zeta,\ \Re{\xi_j} \cdot \Im{\xi_j}=0,
\ee
and
\bel{llb}
|\xi_j| \leq \frac{1}{2} \left( | \eta | + \frac{| \ell |}{| \eta |} \right) + r,\
| \tau_j | \leq | \eta |^2 + \frac{\ell^2}{| \eta |^2} + 2r^2.
\ee
\end{lemma}
\begin{proof}
Let $\eta^{\perp}$ be any non zero $\R^2$-vector, orthogonal to $\eta$
and put $\eta_r^{\perp}=r \eta^{\perp} \slash | \eta^{\perp} |$.
Then, a direct calculation shows that
$$ \xi_j=\frac{1}{2} \left( (-1)^{j+1} + \frac{\ell}{| \eta |^2} \right) \eta + (-1)^ji \eta_r^{\perp},\ \tau_j = \frac{1}{4} \left( (-1)^{j+1} + \frac{\ell}{| \eta |^2} \right)^2 | \eta |^2 - r^2,\ j=1,2, $$
fulfill \eqref{lla}-\eqref{llb}.
\end{proof}

In light of Proposition \ref{pr1} and Lemma \ref{l1} we may now derive the following:

\begin{lemma}
\label{l4}
Assume that $V_j \in W^{2,\infty}(0,T;W^{2,\infty}(\Omega))$, $j=1,2$, fulfill \eqref{(1.2)} and
fix $r \geq r_0=c_0 ( 1+M )>0$, where $M\geq\max_{j=1,2} \| V_j \|_{W^{2,\infty}(0,T;W^{2,\infty}(\Omega))}$ and $c_0$ is the same as in \eqref{l3b}. Pick $\zeta=(\eta,\ell) \in \R^2 \times \R$ with $\eta \neq 0_{\R^2}$, and let $\zeta_j=(\xi_j,\tau_j) \in\C^2\times\R$, $j=1,2$, be given by Lemma \ref{l1}. Then, there is a constant $C>0$ depending only on $T$, $| \omega |$ and $M$, such that for every $k \in \Z$ and $\theta \in [0,2 \pi)$, the function $u_{j,k,\theta}$, $j=1,2$, defined in Proposition \ref{pr1} by substituting $\xi_j$ for $\xi$, satisfies the estimate
\[
\| u_{j,k,\theta} \|_{H^2(0,T;H^2(\Omega '))} \leq C (1+\mathfrak{q}(\zeta,k))^{\frac{13}{2}} \frac{(1+r^2)^3}{r} e^{| \omega| r},\ k \in \Z,\ \theta \in [0,2 \pi),\ r \geq r_0,
\]
with
\[
\mathfrak{q}(\zeta,k)=\mathfrak{q}(\eta,\ell,k)=|\eta|^2+\frac{|\ell |}{|\eta |}+k^2. 
 \]
\end{lemma}

\begin{proof}
In light of \eqref{t2r} we have
\beas
& & \| u_{j,k,\theta} \|_{H^2(0,T;H^2(\Omega '))} \\
& \leq &
\left( \| e^{i \theta x_1} \|_{H^2(0,T;H^2(\Omega '))} + \| w_{j,k,\theta} \|_{H^2(0,T;H^{2}(\Omega '))} \right) \| e^{-i \left((\tau_j+ 4 \pi^2 k^2)t+2 \pi k x_1+x'\cdot\xi_j \right)} \|_{W^{2,\infty}(0,T;W^{2,\infty}(\Omega '))},
\eeas
with 
\[
\| e^{-i \left((\tau_j+4 \pi^2 k^2)t+ 2 \pi k x_1+x'\cdot\xi_j \right)} \|_{W^{2,\infty}(0,T;W^{2,\infty}(\Omega '))}
\leq  (1 + |\tau_j|+ 4 \pi^2 k^2)^2 (1+ | \xi_j |^2+4 \pi^2 k^2 ) e^{|\omega| r},
\]
and
\[ \| e^{i \theta x_1} \|_{H^2(0,T;H^{2}(\Omega '))}
\leq c ( T | \omega |)^{1 \slash 2},
\]
for some positive constant $c$ which is independent of $r$, $\theta$, $\zeta$, $k$, $T$ and $\omega$.
Thus we get the desired result by combining the three above inequalities with \eqref{t2a} and \eqref{llb}.
\end{proof}

\subsection{Proof of Theorem \ref{thm2}}
Let $\zeta=(\eta,\ell)$, $r$ and $\zeta_j=(\xi_j,\tau_j)$, $j=1,2$, be as in Lemma \ref{l4}, fix $k\in \Z$, and put
\[
(k_1,k_2)=\left\{ \begin{array}{ll} ( k \slash 2,-k \slash 2)\ & \textrm{if}\ k\ \textrm{is even} \\ ((k+1) \slash 2,-(k-1) \slash 2) & \textrm{if}\ k\ \textrm{is odd}. \end{array} \right.
\]
Further we pick $\theta \in [0,2 \pi)$ and note $u_j$, $j=1,2$, the optics geometric solution $u_{j,k_j,\theta}$, defined by Lemma \ref{l4}. In light of Lemma \ref{lm-c} there is a unique solution $v \in L^2(0,T;H_{\sharp,\theta}^2(\Omega ')) \cap H^1(0,T;L^2(\Omega '))$ to the boundary value problem
\bel{t1c}
\left\{ \begin{array}{ll}
(-i\partial_t  + \Delta  + V_2) v = 0\quad  &\textrm{in}\; Q'
\\ 
v(0,\cdot )= u_1(0,\cdot ) &\textrm{in}\; \Omega ',
\\ 
v = u_1 &\textrm{on}\ \Sigma_*', 
\end{array}\right.
\ee
in such a way that $u=v-u_1$ is solution to the following system:
\bel{t1d}
\left\{ \begin{array}{ll}
(-i\partial_t  + \Delta  + V_2) u = (V_1-V_2)u_1\quad  &\textrm{in}\; Q'
\\ 
u(0,\cdot )= 0 &\textrm{in}\; \Omega ',
\\ 
u = 0 &\textrm{on}\ \Sigma_*', 
\\
u(\cdot ,1,\cdot )=e^{i\theta}u(\cdot ,0,\cdot ) & \mathrm{on}\ (0,T) \times \omega
\\
\partial_{x_1} u(\cdot ,1,\cdot )=e^{i\theta}\partial_{x_1}u(\cdot ,0,\cdot )& \mathrm{on}\ (0,T) \times \omega. \end{array}\right.
\ee
Therefore we get
\bea \label{eq8}
\int_{Q'}(V_1-V_2)u_1 \overline{u_2} \, dt dx 
= \int_{\Sigma_*'} \partial_\nu u \overline{u_2}\, dt d\sigma (x)
-i \int_{\Omega '} u(T,\cdot ) \overline{u_2(T,\cdot )}\, dx, 
\eea
by integrating by parts and taking into account the quasi-periodic boundary conditions satisfied by $u$ and $u_2$.
Notice from \eqref{t1c}-\eqref{t1d} that $\partial_\nu u= \left( \Lambda^1_{V_2,\theta}-\Lambda^1_{V_1,\theta}\right)(\mathfrak{g}_1)$ and 
$u(T,.)=\left( \Lambda^2_{V_2,\theta}-\Lambda^2_{V_1,\theta}\right)(\mathfrak{g}_1)$, where
\[ 
\mathfrak{g}_1 = \left( u_1{_{| \Sigma_*'}} , u_1(0,.) \right) \in \mathscr{X}'_{0,\theta}.
\]
Thus, putting
\[ 
\beta_k = \left\{ \begin{array}{ll} 0 & \textrm{if}\ k\ \textrm{is even or}\ k \in \R \setminus \Z \\ 4 \pi^2 & \textrm{if}\ k\ \textrm{is odd}, \end{array} \right. 
\]
for all $k \in \Z$, and
\bel{eq8b}
\varrho=\varrho_{k,\theta}=e^{-i \theta x_1} w_1 + e^{i \theta x_1} \overline{w_2} + w_1 \overline{w_2},
\ee
we deduce from \eqref{lla}, \eqref{t2r} and \eqref{eq8} that
\bel{eq9}
\int_{Q'}(V_1-V_2) e^{-i \left((\ell+ \beta_k k) t  + 2 \pi k x_1 + x' \cdot \eta\right)} dt dx =A+B+C, 
\ee
with
\bea
A & = & -\int_{Q'}(V_2-V_1) \varrho(t,x) e^{-i \left((\ell+\beta_k k) t  + 2 \pi k x_1 + x' \cdot \eta \right)} dt dx,
\label{eq9a} 
\\
B & =  &\int_{\Sigma_*'} \left( \Lambda^1_{V_2,\theta}-\Lambda^1_{V_1,\theta}\right)(\mathfrak{g}_1)\overline{u_2}\, dt d\sigma (x), \label{eq9b} 
\\
C & = & -i \int_{\Omega '} \left( \Lambda^2_{V_2,\theta}-\Lambda^2_{V_1,\theta}\right)(\mathfrak{g}_1) \overline{u_2(T,\cdot)}\, dx.
\label{eq9c}
\eea
Next, we introduce
\[ 
V(t,x) = \left\{ \begin{array}{cl} (V_2-V_1)(t,x) & \textrm{if}\ (t,x) \in Q, \\ 0 & \textrm{if}\
(t,x) \in \R^4 \setminus Q, \end{array} \right. 
\]
and 
\[ 
\phi_k(x_1) = e^{i 2 \pi k x_1},\ x_1 \in \R,\ k \in \Z, 
\]
so \eqref{eq9} can be rewritten as
\bel{eq10}
\int_{Q'}(V_1-V_2) e^{-i \left((\ell+ \beta_k k) t  + 2 \pi k x_1+ x' \cdot \eta \right)}\, dt dx
= \left\langle \widehat{V}(\ell+ \beta_k k, \eta) , \phi_k \right\rangle_{L^2(0,1)}, 
\ee
where $\widehat{V}$ stands for the partial Fourier transform of $V$ with respect to  $t \in \R$ and $x' \in \R^2$.
Further, in light of \eqref{t2a} and \eqref{eq8b} it holds true that
\beas
\| \varrho \|_{L^1(Q')} & \leq  &(T | \omega |)^{1 \slash 2} \left(\| w_1\|_{L^2(Q')}+ \| w_2\|_{L^2(Q')} \right) + \| w_1\|_{L^2(Q')}\| w_2\|_{L^2(Q')}
\\
& \leq  & \frac{C}{r} \left( (T | \omega |)^{1 \slash 2} (2+|k_1|+ |k_2|) + \frac{C}{r} (1+|k_1|)(1+|k_2|) \right) 
\\
& \leq  &  c'\left( \frac{1+|k|}{r} \right)^2, 
\eeas
where the constant $c'>0$ depends only on $T$, $| \omega|$ and $M$.
Since $\| V_1-V_2\|_{\infty} \leq 2M$, it follows from this and \eqref{eq9a} upon substituting $c'$ for $4M c'$ in the above estimate that
\bel{eq10a}
| A | \leq \| V_1-V_2\|_{\infty} \| \varrho \|_{L^1(Q')} \leq c'\frac{(1+\mathfrak{q}(\zeta,k))}{r^2},
\ee
where $\mathfrak{q}$ is defined in Lemma \ref{l4}.
Moreover, we have
\[
 |B| \leq \| (\Lambda_{V_2,\theta}^1-\Lambda_{V_1,\theta}^1)(\mathfrak{g}_1) \|_{L^2(\Sigma_*')} \| u_2 \|_{L^2(\Sigma_*')}, 
 \]
by \eqref{eq9b} and
\[
|C| \leq \| (\Lambda_{V_2,\theta}^2-\Lambda_{V_1,\theta}^2)(\mathfrak{g}_1) \|_{L^2(\Omega ')} \| u_2 \|_{L^2(\Omega ')}, 
\]
from \eqref{eq9c}, whence
\bea
|B|+|C| &\leq  & \left( \| (\Lambda_{V_2,\theta}^1-\Lambda_{V_1,\theta}^1)(\mathfrak{g}_1) \|_{L^2(\Sigma_*')}^2 + \| (\Lambda_{V_2,\theta}^2-\Lambda_{V_1,\theta}^2)(\mathfrak{g}_1) \|_{L^2(\Omega ')}^2 \right)^{1 \slash 2} \label{eq10b} \\
& & \times \left( \| u_2 \|_{L^2(\Sigma_*')}^2 +
\| u_2(T,.) \|_{L^2(\Omega ')}^2 \right)^{1 \slash 2} \nonumber \\
& \leq  & \left\| (\Lambda_{V_2,\theta}-\Lambda_{V_1,\theta})(\mathfrak{g}_1) \right\|_{L^2(\Sigma_*') \times L^2(\Omega ')} \| \mathfrak{g}_2 \|_{L^2(\Sigma_*') \times L^2(\Omega ')} \nonumber \\
& \leq & \| \Lambda_{V_2,\theta}-\Lambda_{V_1,\theta} \|_{\B(\mathscr{X}'_{0,\theta},\mathscr{X}'_1)} \| \mathfrak{g}_1 \|_{\mathscr{X}'_{0,\theta}} \| \| \mathfrak{g}_2 \|_{\mathscr{X}'_1}, \nonumber
\eea
where we note
\[
\mathfrak{g}_2=\left( u_2{_{| \Sigma_*'}} , u_2(T,.) \right).
\]
The next step is to use that 
$\| \mathfrak{g}_1 \| _{\mathscr{X}'_{0,\theta}}$ and
$\| \mathfrak{g}_2 \| _{\mathscr{X}'_{0,\theta}}$ are both upper bounded, up to some multiplicative constant depending only on $T$ and $\omega$, by $\| u_j \|_{H^2(0,T;H^{2}(\Omega '))}$.
Therefore \eqref{eq10b} and Lemma \ref{l4} yield
\bel{eq10d}
|B|+|C| \leq C^2 \| \Lambda_{V_2,\theta}^1-\Lambda_{V_1,\theta}^1 \|_{\B(\mathscr{X}'_{0,\theta},\mathscr{X}'_1)}
(1+\mathfrak{q}(\zeta,k))^{13} \frac{(1+r^2)^6}{r^2} e^{2 |\omega| r },\ r \geq r_0.
\ee
Now, putting \eqref{eq9}--\eqref{eq10a} and
\eqref{eq10d} together,  we end up getting that
\bel{eq10e}
\left| \left\langle \widehat{V}(\ell + \beta_k k, \eta) , \phi_k \right\rangle_{L^2(0,1)} \right|
\leq c'' \frac{(1+\mathfrak{q}(\zeta,k))}{r^2} \left( 1 + \gamma (1+\mathfrak{q}(\zeta,k))^{12} (1+r^2)^6 e^{2 |\omega| r} \right),\ r \geq r_0,
\ee
where 
\[
\gamma=\| \Lambda_{V_2,\theta}-\Lambda_{V_1,\theta } \|_{\B(\mathscr{X}'_{0,\theta},\mathscr{X}'_1)}
\] 
and the constant $c''>0$ is independent of $k$, $r$ and $\zeta=(\eta,\ell)$.

The next step is to apply Parseval-Plancherel theorem, getting
\bel{eq11}
\| V_2 - V_1 \|_{L^2(Q')}^2=\| V \|_{L^2(\R \times (0,1)\times \R^2)}^2 = \sum_{k \in \Z} \int_{\R^3} | \hat{\mathfrak{v}}(\zeta,k) |^2 d \zeta,
\ee
where $\hat{\mathfrak{v}}(\zeta,k)=\hat{\mathfrak{v}}(\ell,\eta,k)$ stands for
$\langle \widehat{V}(\ell, \eta) , \phi_k \rangle_{L^2(0,1)}$ for all $(\zeta,k) \in \R^3 \times \Z$. By splitting $\int_{\R^3} | \hat{\mathfrak{v}}(\zeta,k) |^2 d \zeta$, $k \in \Z$, into the sum 
$\int_{\R^3} | \hat{\mathfrak{v}}(\ell,\eta,2k) |^2 d \ell d \eta
+ \int_{\R^3} | \hat{\mathfrak{v}}(\ell,\eta,2k+1) |^2 d \ell d \eta$ and performing the change 
of variable $\ell'=\ell-(2k+1)$ in the last integral, we may actually rewrite \eqref{eq11} as
\bel{eq11a}
\| V_2 - V_1 \|_{L^2(Q')}^2  =  \sum_{k \in \Z} \int_{\R^3} | \hat{\mathfrak{v}}(\ell+ \beta_k k,\eta,k) |^2 d \ell d \eta 
= \int_{\R^4} | \hat{\mathfrak{v}}(\ell+ \beta_k k,\eta,k) |^2 d \ell d \eta d \mu(k),
\ee
where $\mu=\sum_{n \in \Z} \delta_n$.
Putting $B_\rho=\{(\zeta,k) \in \R^3 \times \Z,\ | (\zeta,k) | < \rho \}$ for some $\rho>0$ we shall make precise below, we treat $\int_{B_{\rho}} | \hat{\mathfrak{v}}(\ell+ \beta_k k,\eta,k) |^2 d \ell d \eta d \mu(k)$ and $\int_{\R^4 \setminus B_{\rho}} | \hat{\mathfrak{v}}(\ell+ \beta_k k,\eta,k) |^2 d \ell d \eta d \mu(k)$ separately.
We start by examining the last integral. 
To do that we first notice that $(\ell,\eta,k) \mapsto |(\ell + \beta_k k,\eta,k)|$ is a norm in $\mathbb{R}^4$ so we may find
a constant $C_1>0$ such that the estimate
\[
|(\ell,\eta,k)| \leq C_1 |(\ell + \beta_k k,\eta,k)|. 
\]
holds for every $(\ell,\eta,k) \in \R^4$.
As a consequence we have
\beas
\int_{\R^4 \setminus B_{\rho}} | \hat{\mathfrak{v}}(\ell+ \beta_k k,\eta,k) |^2 d \ell d \eta d \mu(k) 
& \leq  & \frac{1}{\rho^2} \int_{\R^4 \setminus B_{\rho}} |(\ell,\eta,k)|^2  |\hat{\mathfrak{v}}(\ell+ \beta_k k,\eta,k) |^2 d \ell d \eta d \mu(k) \\
& \leq  & \frac{C_1^2}{\rho^2} \int_{\R^4 \setminus B_{\rho}} |(\ell + \beta_k k,\eta,k)|^2  |\hat{\mathfrak{v}}(\ell+ \beta_k k,\eta,k) |^2 d \ell d \eta d \mu(k) \\
& \leq & \frac{C_1^2}{\rho^2} \int_{\R^4} (1+|(\ell + \beta_k k,\eta,k)|^2)  |\hat{\mathfrak{v}}(\ell+ \beta_k k,\eta,k) |^2 d \ell d \eta d \mu(k). 
\eeas
The change of variable $\ell'= \ell+ \beta_k k$ in the last integral then yields
$$
\int_{\R^4 \setminus B_{\rho}} | \hat{\mathfrak{v}}(\ell+ \beta_k k,\eta,k) |^2 d \ell d \eta d \mu(k)
\leq  \frac{C_1^2}{\rho^2} \int_{\R^4} (1+|(\zeta,k)|^2)  |\hat{\mathfrak{v}}(\zeta,k) |^2 d \zeta d \mu(k) 
\leq \frac{C_1^2}{\rho^2} \| V \|_{H^1(Q')}^2,
$$
so we end up getting that
\bel{eq11b}
\int_{\R^4 \setminus B_{\rho}} | \hat{\mathfrak{v}}(\ell+ \beta_k k,\eta,k) |^2 d \ell d \eta d \mu(k) \leq \frac{4 C_1^2 M^2}{\rho^2}.
\ee
Further, we introduce $\mathcal{C}_\rho=\{ (\zeta,k) \in \R^4,\ | \eta | <\rho^{-1} \}$ in such a way that the integral
$\int_{ B_{\rho} \cap \mathcal{C}_\rho} | \hat{\mathfrak{v}}(\ell+ \beta_k k,\eta,k) |^2  d \ell d \eta d \mu(k)$ is upper bounded by
\beas
\int_{ \mathcal{C}_{\rho} } | \hat{\mathfrak{v}}(\ell+ \beta_k k,\eta,k) |^2 d \ell d \eta d \mu(k)
& \leq  & \frac{\pi}{\rho^2} \sup_{|\eta| \leq \rho^{-1}} \int_{\R^2} | \hat{\mathfrak{v}}(\ell+ \beta_k k,\eta,k) |^2 d \ell d \mu(k) \nonumber \\
& \leq  & \frac{\pi}{\rho^2} \sup_{|\eta| \leq \rho^{-1}} \sum_{k \in \Z} \int_{\R} | \hat{\mathfrak{v}}(\ell+ \beta_k k,\eta,k) |^2 d \ell \nonumber \\
& \leq  & \frac{\pi}{\rho^2} \sup_{|\eta| \leq \rho^{-1}} \sum_{k \in \Z} \| \hat{\mathfrak{v}}(.,\eta,k) \|_{L^2(\R)}^2,
\eeas
giving
\bel{eq11c}
\int_{ B_{\rho} \cap \mathcal{C}_\rho} | \hat{\mathfrak{v}}(\ell+ \beta_k k,\eta,k) |^2  d \ell d \eta d \mu(k) 
\leq \frac{\pi}{\rho^2} \| V \|_{L_{x'}^{\infty}(\R^2;L_{t,x_1}^2(\R \times (0,1)))}^2 \leq \frac{4\pi M^2}{\rho^2},
\ee
and
\[
\mathfrak{q}(\zeta,k) \leq 3 \rho^2,\ (\zeta,k) \in B_\rho \cap (\R^4 \setminus \mathcal{C}_\rho),\ \rho \geq 1.
\]
From \eqref{eq10e} and the above estimate then follows that
\[
| \hat{\mathfrak{v}}(\ell+ \beta_k k,\eta,k) | \leq c''\frac{\rho^2}{r^2} \left( 1 + \gamma \rho^{24} r^{12} e^{2 |\omega | r} \right),\ (\zeta,k) \in B_\rho \cap (\R^4 \setminus \mathcal{C}_\rho),\ \rho \geq 1,\ r \geq \max(1,r_0),
\]
whence
\bel{eq11d}
\int_{ B_{\rho} \cap (\R^4 \setminus \mathcal{C}_\rho)} | \hat{\mathfrak{v}}(\ell+ \beta_k k,\eta,k) |^2 d \ell d \eta d \mu(k) 
\leq c''\frac{\rho^{8}}{r^4} \left( 1+ \gamma^2 \rho^{48} r^{24} e^{4 | \omega | r} \right),\ \rho \geq 1,\ r \geq \max(1,r_0),
\ee
upon eventually substituting $c''$ for some suitable algebraic expression of $c''$.

Last, putting \eqref{eq11a}--\eqref{eq11d} together we find out that
\bel{eq12}
\| V_2 - V_1 \|_{L^2(Q')}^2 \leq C_2 \left( \frac{1}{\rho^2} + \frac{\rho^8}{r^4} + \gamma^2 \rho^{56} r^{20}  e^{4 | \omega | r} \right),\ \rho \geq 1,\ r \geq \max(1,r_0),
\ee
where the constant $C_2>0$ depends only on $T$, $\omega$ and $M$. By taking
$r=r_1=\frac{1}{4 | \omega |} \ln \gamma^{-1}$ and $\rho=r_1^{2 \slash 5}$ in \eqref{eq12}, which is permitted since $r_1> \max(1,r_0)$ from \eqref{(1.11)}, we find out that 
\bel{eq13}
\| V_2 - V_1 \|_{L^2(Q')}^2 \leq C_3 \left( 1 + \gamma \left( \ln \gamma^{-1} \right)^{16 \slash 5} \right) \left( \ln \gamma^{-1} \right)^{-4 \slash 5},
\ee
where $C_3$ is another positive constant depending only on $T$, $\omega$ and $M$. Now, since 
$\sup_{0<\gamma \leq \gamma_*} \left( \gamma \left( \ln \gamma^{-1} \right)^{16 \slash 5}\right)$ is just another constant depending only on $T$, $\omega$ and $M$, then \eqref{(1.10)} follows readily from \eqref{eq13}.


\bigskip


\end{document}